\documentclass[12pt]{article}
\usepackage{natbib}
\usepackage{dcolumn}
\usepackage{graphicx}
\usepackage{amsmath}
\usepackage{amsfonts}
\usepackage{amsopn}
\usepackage{amsthm}
\usepackage[figuresright]{rotating}
\usepackage{amsmath,amsthm,amsfonts,epsfig,amssymb,natbib,eucal,eufrak}
\usepackage{enumitem}
\usepackage{color}
\newtheorem{thm}{Theorem}
\newtheorem{lem}{Lemma}

\newtheorem{cor}{Corollary}

\renewcommand{\baselinestretch}{1.8}

\begin{document}
\begin{center}
{\Large\bf Monotonicity in the Sample Size of the Length of Classical Confidence Intervals}
\end{center}
\vspace{.5cm}
\centerline{Abram M. Kagan${}^{a}$ and Yaakov Malinovsky${}^{b, *}$}
\begin{center}
${}^{a}$ Department of Mathematics, University of Maryland, College Park, MD 20742, USA \\
${}^{b}$ Department of Mathematics and Statistics , University of Maryland Baltimore County, Baltimore, MD 21250, USA
\end{center}
\centerline{${}^{a}$   {\it email}: amk@math.umd.edu}
\centerline{${}^{b}$  {\it email}: yaakovm@umbc.edu, ${}^{*}$Corresponding author}

\vspace{.5cm} \centerline{Summary}
It is proved that the average length of standard confidence intervals for parameters of gamma and normal distributions
monotonically decrease with the sample size. The proofs are based on fine properties of the classical gamma function.\\
\noindent {\it Key words:} Gamma function; Location and scale parameters; Stochastic monotonicity

\section{Introduction and Lemmas}
In recent issues of the Bulletin of the IMS (see \cite{Shi2008}, \cite{DasGupta2008}), a discussion was held on the behavior of standard estimators of parameters as functions of the sample size $n$.
If $R(\widetilde{\theta}_{n}, \theta)$ is the risk of an estimator $\widetilde{\theta}_{n}$ constructed from a sample of size $n$, a very desirable property
of $\widetilde{\theta}_{n}$ would be
\begin{equation}
\label{eq:rschyi}
R(\widetilde{\theta}_{n+1}, \theta)\leq R(\widetilde{\theta}_{n}, \theta)
\end{equation}
for all $\theta$. Unfortunately, even when \eqref{eq:rschyi} holds for a class of estimators and/or families, it can be difficult to prove it.\\
\noindent
One of few examples of classical estimators with a monotonically (in $n$) decreasing risk is the Pitman estimator of a location parameter.
Let $\left(X_1,\ldots,X_n\right)$ be a sample from population $F(x-\theta)$ and let ${t}_{n}={t}_{n}\left(X_1,\ldots,X_n\right)$
be the Pitman estimator corresponding to an (invariant) loss function $L(\widetilde{{\theta}}, \theta)=L(\widetilde{{\theta}}-\theta)$.
The corresponding risk $R(\widetilde{{\theta}}_{n})$ of any equivariant estimator $\widetilde{{\theta}}_{n}$ is constant in $\theta$ and
by the very definition of $t_{n}$, $R(t_{n})\geq R(t_{n+1})$ for any $F$.
A deeper result holds for the Pitman estimator corresponding to the quadratic loss $L(\widetilde{\theta}-\theta)=(\widetilde{\theta}-\theta)^{2}$.
If $\displaystyle \int x^{2}dF(x)<\infty$, then for any $n$, $Var(t_{n})<\infty$ and
\begin{equation}
\label{eq:RiskP}
nVar(t_{n})\geq (n+1)Var(t_{n+1}).
\end{equation}
\noindent
The proof of \eqref{eq:RiskP} in \cite{Kagan2011} is based on a lemma of general interest from \cite{Artstein2004}.
The inequality was used in studying a geometric property of the sample mean in \cite{Kagan2009}.

Turning to the interval estimation of parameters, one finds a very natural loss function, namely the length of a confidence interval.
Here we study the risk, i.\,e., the average length of the standard confidence intervals for the scale parameter
$\beta$ of a gamma distribution $Gamma\left(a, \beta\right)$ and for the mean $\mu$ and variance $\sigma^{2}$ of a normal distribution $N(\mu, \sigma^{2} )$.
Though our results are new, to the best of our knowledge, their interest is more methodological than applied. Notice, however, that the distributions we study
are often used in different applications.

It is proved that the average length of the standard confidence interval of a given level $1-\alpha$ monotonically decreases with the sample size
$n$. Though the monotonicity seems a very natural property, the proofs are based on fine properties of the gamma function and are nontrivial.

We write $X\sim Gamma\left(a, \beta\right)$ if the probability density function of $X$ is
\begin{equation}
f\left(x; a,\beta\right)=\frac{1}{\beta^{a}\Gamma(a)}x^{a-1}e^{-x/\beta},\,\,\, x>0, a>0, \beta>0.
\end{equation}

\begin{lem}
\label{lem:lemma1}
Let $F$ and $G$ be distribution function with densities $f$ and $g$ that are positive
and continuous on an open interval $I = (a, b)$, $-\infty\leq a<b\leq \infty$ and are zero off the interval.
Suppose the following condition holds.
\begin{enumerate}
\item[(C)]
There are numbers $c_1$ and $c_2$ in the interval $I$ with $c_1 < c_2$ such that $f(x) > g(x)$
for $x\in(a, c_1)\cup(c_2, b)$ and $f(x) < g(x)$ for $c_1 < x < c_2$.
\end{enumerate}
Then there is a unique $x_0$ such that $F(x) > G(x)$ for $a < x < x_0$ and $F(x) < G(x)$ for
$x_0 < x < b$.
This implies that $F^{-1}(u) < G^{-1}(u)$ for $u < u_0$ and $F^{-1}(u) > G^{-1}(u)$ for
$u > u_0$ where $u_0 = F(x_0) = G(x_0)$.
\end{lem}
\begin{proof}
One has $F(x) > G(x)$ for $a <x\leq c_1$ and $1-F(x)> 1-G(x)$ (and thus $F(x) < G(x)$)
for $c_2\leq x < b$. By the intermediate value theorem there is a point $x_0$ between $c_1$ and $c_2$
such that $F(x_0) = G(x_0)$. This point $x_0$ is unique. Indeed, if there were two such points,
say $x_1$ and $x_2$ with $x_1 < x_2$, we have $F(x_1)-G(x_1) = 0$ and $F(x_2)-G(x_2) = 0$ and
Rolle's Theorem yields $f(x)-g(x)=0$ for some $x \in (x_1, x_2)$ contradicting $f(x) < g(x)$ for
$c_1 < x < c_2$.
\end{proof}
Condition $(C)$ is satisfied if the log-likelihood ratio $\displaystyle r(x)=\log\left(f(x)/g(x)\right),\,\,\,a<x<b$
is strictly convex and $\lim\inf_{x\rightarrow a{+0}}r(x)>0$, $\lim\sup_{x\rightarrow b{-0}}r(x)>0$.

Lemma \ref{lem:lemma1} suggested by an anonymous referee is a general version of the authors' original lemma, which is a direct corollary.
\begin{cor}
\label{eq:cor}
If $X_{1}\sim Gamma\left(a_{1}, \beta_{1}\right)$, $X_{2}\sim Gamma\left(a_{2}, \beta_{2}\right)$ with $a_1<a_2,\,\, \beta_1>\beta_2$, then exists a unique  $x^{*}=x^{*}\left(a_1, a_2, \beta_1, \beta_2\right)$
such that the distribution functions $F_{1}$ of $X_1$, and $F_{2}$ of $X_2$
have the following properties:
\begin{equation}
F_{1}\left(x\right)>F_{2}\left(x\right)\,\,\, for\,\,\, x<x^{*}\,\,\,
\text{and}\,\,\,
F_{1}\left(x\right)<F_{2}\left(x\right)\,\,\, for\,\,\, x>x^{*}.
\end{equation}
In particular, if $\alpha^{*}=\alpha^{*}(a_1, a_2, \beta_1, \beta_2)=F_{1}(x^{*})=F_{2}(x^{*})$ and
${\displaystyle  \gamma_{a_{i}, \beta_{i};\, \alpha}}$ is the quantile of
order $\alpha$ of $Gamma\left(a_{i}, \beta_{i}\right), i=1,2$, then
for $\alpha<\alpha^{*}$, $\displaystyle  \gamma_{a_{1}, \beta_{1};\, \alpha}<\gamma_{a_{2}, \beta_{2};\, \alpha}$
and  $\displaystyle  \gamma_{a_{1}, \beta_{1};\, 1-\alpha}>\gamma_{a_{2}, \beta_{2};\, 1-\alpha}$.
\end{cor}

For a special case of semi-integers $a_1, a_2$ (i.\,e., for the chi-squared distribution)
the result of Corollary \ref{eq:cor} was obtained in \cite{Sz2003} by different arguments.\\
The next lemma deals with a useful property of the classical gamma function.
\begin{lem}
\label{lem:2}
For any $x>0$,
\begin{equation}
\label{eq:G}
\sqrt{x+{1}/{4}}<\frac{\Gamma(x+1)}{\Gamma(x+1/2)}<\sqrt{x+{1}/{2}}.
\end{equation}
\end{lem}
\begin{proof}
For an integer $x$, \eqref{eq:G} was proved in \cite{Lorch1984} and for an arbitrary $x$ in \cite{Laforgia1984}.
\end{proof}
Many useful inequalities for the gamma function are in \cite{Laforgia2011}.
We shall need \eqref{eq:G} for semi-integer $x$.
\begin{figure}
\centerline{\includegraphics[ height=2.8 in]{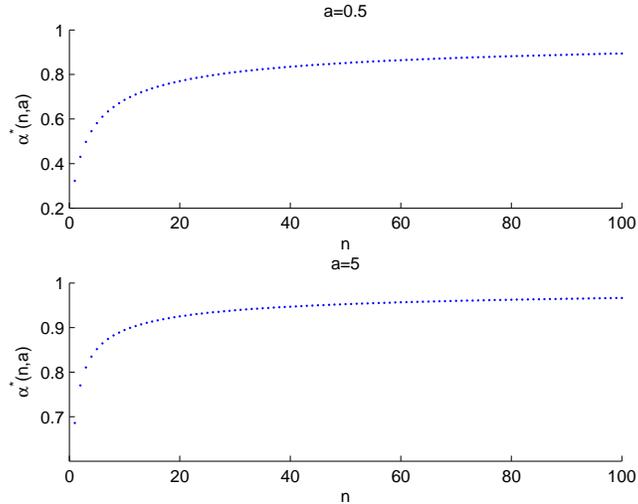}}
\caption{\small The critical values of $\alpha^{*}$ as a function of $n$ }
\label{fig:comp0}
\end{figure}
\section{Mean Length of Confidence Intervals}
\subsection{Confidence Interval for the Scale Parameter of Gamma Distribution}
Let now $\left(X_1,X_2,\ldots,X_n\right)$ be a sample from a population $Gamma\left(a,\beta\right)$ with a known shape parameter $a$
and a scale parameter $\beta$ to be estimated.
The sum $\sum_{i=1}^{n}X_i \sim Gamma\left(na,\beta\right)$ is a sufficient statistics for $\beta$ and
the ratio $\sum_{i=1}^{n}X_i/\beta \sim Gamma\left(na, 1\right)$ is a pivot leading to the standard confidence interval
for $\beta$ of level $1-\alpha$,
$\displaystyle
\left(\frac{n\overline{X}_{n}}{\gamma_{na;\, 1-\alpha/2}},\,\,\, \frac{n\overline{X}_{n}}{\gamma_{na;\, \alpha/2}} \right)
$. Its average length is
$\displaystyle L_{n}=\beta\left(\frac{1}{\gamma_{na;\, \alpha/2}/(na)}-\frac{1}{\gamma_{na;\, 1-\alpha/2}/(na)}\right),$
where $ \displaystyle \gamma_{na;\, \alpha/2}$ is the quantile of order $\alpha/2$ of $Gamma\left(na, 1\right)$.

\noindent
If $G_{i}(x)$ is the distribution function of $X_{i}\sim Gamma(a_{i}, 1), \,i=1,2$, then for $a_1<a_2$, $G_{1}(x)>G_{2}(x)$
(equivalently, $X_1$ is stochastically smaller than $X_2$). Therefore, $ \displaystyle \gamma_{a_{1};\, \alpha}<\gamma_{a_{2};\, \alpha}$
for all $ \displaystyle \alpha,\, 0<\alpha<1$. In particular,
\begin{equation}
\label{eq:in}
\displaystyle
\gamma_{{na};\, \alpha}<\gamma_{{(n+1)a};\, \alpha}.
\end{equation}
The quantile of order $\alpha$ of $Gamma\left(na, 1/(na)\right)$ is $\displaystyle \gamma_{na;\, \alpha}/(na)$ and its relation to the quantile
$\displaystyle \gamma_{(n+1)a;\, \alpha}/((n+1)a)$ differs from \eqref{eq:in}. The following result holds.
\begin{thm}
\label{th:main}
For $\alpha<\alpha^{*}(n,a)$,\,\,$L_{n+1}<L_{n}$.
\end{thm}
\begin{proof}
By virtue of Corollary \ref{eq:cor} applied to the case of
$a_1=na, \beta_1=1/(na), a_2=(n+1)a, \beta_2=1/((n+1)a)$ one gets
\begin{align}
\label{eq:777}
&\frac{1}{\gamma_{(n+1)a;\,\alpha/2}/((n+1)a)}<\frac{1}{\gamma_{na;\, \alpha/2}/(na)}\,\,\,\text{and}\,\,\,
\frac{1}{\gamma_{(n+1)a;\,1-\alpha/2}/((n+1)a)}>\frac{1}{\gamma_{na;\,1-\alpha/2}/(na)},
\end{align}
for $\alpha<\alpha^{*}(n, a)$.
The result follows immediately from \eqref{eq:777}.
\end{proof}
As a function of $n$ for a given $a$, $\alpha^{*}(n,a)$ grows very fast (see Figure \ref{fig:comp0}).

Note that Theorem \ref{th:main} also holds for an asymmetric confidence interval. Namely, let $\alpha_1+\alpha_2=\alpha$,
then the average length of the confidence interval
$\displaystyle \left(\frac{n\overline{X}_{n}}{\gamma_{na;\,1-\alpha_{2}}},\,\,\, \frac{n\overline{X}_{n}}{\gamma_{na;\,\alpha_{1}}} \right)$
is a decreasing function of $n$.

A standard (one-sided) lower confidence bound of level $1-\alpha$ for the parameter $\beta$ is
$\displaystyle \frac{n\overline{X}_{n}}{\gamma_{na;\,1-\alpha}}$.
The statistician is interested in having (for a given level $1-\alpha$) a larger lower bound.
From Corollary \ref{eq:cor} for $\alpha<\alpha*$ it follows that
$\displaystyle E\left(\frac{n\overline{X}_{n}}{\gamma_{na;\,1-\alpha}}\right)=\frac{\beta}{\gamma_{na;\,1-\alpha}/(na)}$
is an increasing function of $n$. Similarly, for an upper confidence bound $\displaystyle \frac{n\overline{X}_{n}}{\gamma_{na;\,\alpha}}$
of level $1-\alpha$, $\displaystyle E\left(\frac{n\overline{X}_{n}}{\gamma_{na;\,\alpha}}\right)$ is decreasing function of $n$.
\subsection{Confidence Interval for the Normal Variance}
Let now $\left(X_1,X_2,\ldots,X_n\right)$
be a sample from a normal population $N\left(\mu, \sigma^2\right)$ with $\mu$ and $\sigma^2$ as parameters. The standard confidence interval of level $\left(1-\alpha\right)$
for $\sigma^2$ is
\begin{equation}
\label{eq:chiinterval}
\left(\frac{(n-1)S_{n}^2}{\chi^{2}_{n-1;\,1-\alpha/2}},\,\,\,\frac{(n-1)S_{n}^2}{\chi^{2}_{n-1;\,\alpha/2}}\right),
\end{equation}
where $S_{n}^2$ is the sample variance and $\chi^{2}_{n-1;\,\alpha/2}$ is the quantile of
order $\alpha/2$ of chi-square distribution with $n-1$ degrees of freedom. The average length of the interval \eqref{eq:chiinterval}
is \\$L_{n}=\displaystyle \sigma^2 \left(\frac{1}{\chi^{2}_{n-1;\,\alpha/2}/(n-1)}-\frac{1}{\chi^{2}_{n-1;\,1-\alpha/2}/(n-1)}\right).$
If $X\sim \chi^{2}_{d}$, then $X/d \sim Gamma\left(\frac{d}{2}, \frac{2}{d}\right)$
and again Corollary \ref{eq:cor} is applicable.  Thus, for $\alpha<\alpha^{*}(n)$, monotonicity of $L_n$ holds, $L_n>L_{n+1}.$
A table of the values of $\alpha^{*}(n)$ can be found in \cite{Sz2003}.
For the sake of completeness a graph of $\alpha^{*}(n)$ is drawn in Figure \ref{fig:comp}.
\renewcommand{\baselinestretch}{1.1}
\begin{figure}
\centerline{\includegraphics[ height=2.8 in]{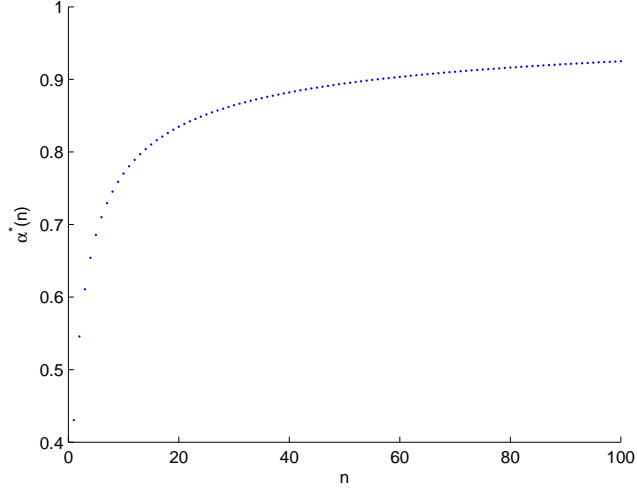}}
\caption{\small The critical values of $\alpha^{*}$ as a function of $n$}
\label{fig:comp}
\end{figure}

\subsection{Confidence Interval for the Normal Mean}
The standard Student confidence interval of level $1-\alpha$ for $\mu$ is
\begin{equation}
\label{eq:888}
\left(\overline{X}_{n}-t_{n-1;\, \alpha/2}\frac{S_{n}}{\sqrt{n}}\,,\,\,\, \overline{X}_{n}+t_{n-1;\, \alpha/2}\frac{S_{n}}{\sqrt{n}}\right),
\end{equation}
where $\displaystyle t_{d;\, \alpha}$ is the quantile of order $1-\alpha$  of the
Student distribution with $d$ degrees of freedom. The average length $L_{n}$ of \eqref{eq:888} is easily calculated,
\begin{equation}
L_n=2\sqrt{2}\sigma\,\,\, t_{n-1;\, \alpha/2}E_n,\,\,\, \text{where}\,\,\,E_n=\frac{\Gamma(n/2)}{\Gamma((n-1)/2)\sqrt{n(n-1)}}.
\end{equation}
The quantile $\displaystyle t_{d;\, \alpha}$ decreases monotonically in $d$ for any $\alpha<1/2$. This known fact (see, e.g.,  \cite{Ghosh1973})
follows from Lemma \ref{lem:lemma1} with $c_2=\infty$ due to the following properties of the probability density function $\displaystyle f_{d}(x)$
of the Student distribution with $d$ degree of freedom:
\begin{align*}
&f_{d}(x)<f_{d+1}(x),\,\,\,0<x<x_0\\
&f_{d}(x)>f_{d+1}(x),\,\,\,x>x_0,
\end{align*}
for some  $x_0>0$.
\\
To prove that $E_{n}>E_{n+1}$ take the left inequality from Lemma \ref{lem:2}. One has
\begin{align}
\label{eq:L}
&E_n=\frac{1}{\sqrt{n(n-1)}}\frac{\Gamma(n/2)}{\Gamma((n-1)/2)}=\frac{1}{\sqrt{n(n-1)}}\frac{\Gamma((n-2)/2+1)}{\Gamma((n-2)/2+1/2)}
>\frac{1}{\sqrt{n(n-1)}}\sqrt{\frac{n-2}{2}+\frac{1}{4}}.
\end{align}
The right inequality from Lemma \ref{lem:2} implies
\begin{align}
\label{eq:U}
&E_{n+1}=\frac{1}{\sqrt{n(n+1)}}\frac{\Gamma((n+1)/2)}{\Gamma(n/2)}=
\frac{1}{\sqrt{n(n+1)}}\frac{\Gamma((n-1)/2+1)}{\Gamma((n-1)/2+1/2)}<
\frac{1}{\sqrt{n(n+1)}}\sqrt{\frac{n-1}{2}+\frac{1}{2}}.
\end{align}
Now comparing the right hand sides of \eqref{eq:L} and \eqref{eq:U} results in $E_{n}>E_{n+1}$ for $n>3$.
For $n=2,3$, the inequalities $E_2>E_3>E_4$ follow from the explicitly calculated values of $E_2, E_3$, and $E_4$.
\subsection{Miscellaneous Results}
Here we present two examples of families with one-dimensional parameter and univariate sufficient statistics whose distributions
in samples of size $n$ and $n+1$ belong to the same type. In the first example monotonicity of the length of standard confidence interval
follows from Lemma \ref{lem:lemma1}, while in the second it is proved by simple direct calculations.\\
{\it Example $1$}. Let $\left(X_1,\ldots,X_n\right)$ be a sample from Pareto distribution  with probability density function
$$\displaystyle f\left(x; \theta\right)=\frac{\theta-1}{x^{\theta}},\,\,\,x\geq1$$
with $\displaystyle \theta>1$ as a parameter. The sufficient statistic for $\displaystyle \theta$ is $\displaystyle S_{n}=\sum_{i=1}^{n}\log(X_i)$.
The pivot $\displaystyle (\theta-1)S_n$ has a gamma distribution $\displaystyle Gamma\left(n, 1\right)$. The standard confidence interval of level $1-\alpha$ for $\theta$ is
\begin{align*}
&\left(1+\frac{\gamma_{n;\,\alpha/{2}}}{S_n},\,\,\,1+\frac{\gamma_{n;\,1-\alpha/{2}}}{S_n}\right)
\end{align*}
and its average length is
\begin{align*}
&L_{n}=\left(\gamma_{n;\,1-\alpha/{2}}-\gamma_{n;\,\alpha/{2}}\right)E\left(\frac{1}{S_{n}}\right)
=(\theta-1)\left(\frac{\gamma_{n;\,1-\alpha/{2}}-\gamma_{n;\,\alpha/{2}}}{n}\right)\frac{n}{n-1}.
\end{align*}
Due to \eqref{eq:777},
$\displaystyle \frac{\gamma_{n+1;\,1-\alpha/{2}}}{n+1}-\frac{\gamma_{n+1;\,\alpha/{2}}}{n+1}< \frac{\gamma_{n;\,1-\alpha/{2}}}{n}-\frac{\gamma_{n;\,\alpha/{2}}}{n}$.
Furthermore,  $\displaystyle \frac{n+1}{n}<\frac{n}{n-1}$
so that $\displaystyle L_{n+1}<L_{n}$. \hfill $\square$
\\
\\
{\it Example $2$}. Let $\left(X_1,\ldots,X_n\right)$ be a sample from a uniform distribution $\displaystyle U(0, \theta)$ on $(0, \theta)$
with $\displaystyle \theta>0$ as a parameter. The sufficient statistic for $\theta$ is $\displaystyle M_{n}=\max\left(X_1,\ldots,X_n\right)$
and the standard confidence interval of level $1-\alpha$ for $\theta$ is $\displaystyle \left(M_{n}, M_{n}/\alpha^{1/n}\right)$.
The average length $\displaystyle L_{n}$ is
\begin{align*}
&L_{n}=\frac{n}{n+1}\theta \left(\frac{1}{\alpha^{1/n}}-1\right)
\end{align*}
and simple calculations show that $\displaystyle L_{n+1}<L_{n}$ for any $n\geq1$ and $\alpha<1$.
\section{An Open Problem}
Let $\left(X_1,X_2,\ldots,X_n\right)$ be a sample from a population with a distribution $\displaystyle F\left(x; \theta\right)$ given by
$$\displaystyle dF\left(x; \theta\right)=e^{\theta x-\psi(\theta)}dF(x),\,\,\,\theta \in \Theta.$$
In other words, $\displaystyle F\left(x; \theta\right)$ belongs to a natural exponential family (NEF) with  generator $\displaystyle F$.
The sum $\displaystyle T_{n}=\sum_{i=1}^{n}X_{i}$ is a complete sufficient statistic for $\displaystyle \theta$.
Let $\displaystyle \delta_{n;\, \alpha}\left(\theta\right)$ be the quantile of order $\displaystyle \alpha$ of the distribution of
$\displaystyle T_{n}$. Since the latter has the monotone likelihood ratio property,  $\displaystyle \delta_{n;\, \alpha}\left(\theta\right)$
is monotone in $\displaystyle \theta$.

The random variable
\begin{align*}
&h\left(T_{n}; \theta\right)=
\left\{
\begin{array}{ccc}
   1, &  & \delta_{n;\, \alpha/2}\left(\theta\right)<T_{n}<\delta_{n;\, 1-\alpha/2}\left(\theta\right) \\
   0, &  &\text{otherwise}
 \end{array}
 \right.
\end{align*}
is a pivot leading to a confidence interval of level $1-\alpha$ for $\displaystyle \theta$,
\begin{equation}
\label{eq:open}
\left(\delta^{-1}_{n;\, 1-\alpha/2}(T_{n}),\,\,\,\ \delta^{-1}_{n;\, \alpha/2}(T_{n})\right).
\end{equation}
One expects that the mean length of \eqref{eq:open} decreases monotonically in $n$. To the best of our knowledge,
this is proved only for a few special $F$. A general result would be of a methodological interest, at the very least.
\\
\\
{\bf Acknowledgment}\\
The authors would like to thank the referee for careful reading the manuscript and very helpful suggestions.


\begin{thebibliography}{}
\bibitem[\protect\citeauthoryear{ Artstein \it{et~al.}}{2004}]{Artstein2004} Artstein, S., Ball, K. M., Barthe, F., Naor, A., 2004.
Solution of Shannon's problem on the monotonicity of entropy.
J. Amer. Math. Soc. 17, 975--982.

\bibitem[\protect\citeauthoryear{DasGupta}{2008}]{DasGupta2008} DasGupta, A., 2008.
Letter to the Editors. IMS Bulletin 36, No.\,\,6, 16.

\bibitem[\protect\citeauthoryear{Ghosh}{1973}]{Ghosh1973} Ghosh, B. K., 1973.
Some monotonicity theorems for $\chi^2$, F and t distributions with applications.
J. Royal Stat. Soc., Ser. B 35, 480--492.


\bibitem[\protect\citeauthoryear{Kagan and Yu}{2009}]{Kagan2009} Kagan, A. M., Yu, T., 2009.
A geometric property of the sample mean and residuals.
Statist. Probab. Lett. 79, 1409--1413.

\bibitem[\protect\citeauthoryear{Kagan \it{et~al.}}{2011}]{Kagan2011} Kagan, A. M., Yu, T., Barron, A., Madiman, M., 2011.
Contribution to the theory of Pitman estimators. Preprint.

\bibitem[\protect\citeauthoryear{Laforgia}{1984}]{Laforgia1984} Laforgia, A., 1984.
Further inequalities for the gamma function. Math. Comp. 42, 597--600.

\bibitem[\protect\citeauthoryear{Laforgia and Natalini}{2011}]{Laforgia2011} Laforgia, A., Natalini, P., 2011.
Some inequalities for the ratio of gamma functions. J. Ineq. Spec. Fun. 2, 16--26.

\bibitem[\protect\citeauthoryear{Lorch}{1984}]{Lorch1984} Lorch, L., 1984.
Inequalities for ultraspherical polynomials and the gamma function. J. Approx. Theory 40, 115--120.

\bibitem[\protect\citeauthoryear{Shi}{2008}]{Shi2008} Shi, N-Z., 2008. Letter to the Editors.
IMS Bulletin 36, No.\,\,4, 4.

\bibitem[\protect\citeauthoryear{ Sz\'{e}kely and Bakirov}{2003}]{Sz2003} Sz\'{e}kely, G. J., Bakirov, N. K., 2003.
Extremal probabilities for Gaussian quadratic forms. Probab. Theory Relat. Fields 126, 184--202.

\end{thebibliography}
\end{document}